\documentclass[12pt]{amsart}

\usepackage{amssymb}
\usepackage{mathtools}
\usepackage{bm}
\usepackage{graphicx}
\usepackage{psfrag}
\usepackage{color}
\usepackage{url}
\usepackage[linesnumbered, ruled, vlined]{algorithm2e}
\usepackage{algpseudocode}
\usepackage{fancyhdr}
\usepackage{xy}
\usepackage{dirtytalk}
\usepackage{float}
\usepackage{stmaryrd}
\input xy
\xyoption{all}

\newtheorem{theorem}{Theorem}[section]
\newtheorem{proposition}[theorem]{Proposition}
\newtheorem{remark}[theorem]{Remark}
\newtheorem{lemma}[theorem]{Lemma}
\newtheorem{cor}[theorem]{Corollary}

\theoremstyle{definition}
\newtheorem{definition}[theorem]{Definition}
\newtheorem{example}[theorem]{Example}
\numberwithin{equation}{section}

\newcommand\nn{\mathbb{N}}
\newcommand\qq{\mathbb{Q}}

\newcommand\zz{\mathbb{Z}}

\begin{document}
	
	\mbox{}
	\title{On the Sets of Lengths of Puiseux Monoids Generated by Multiple Geometric Sequences}
	\author{Harold Polo}
	\address{Mathematics Department\\University of Florida\\Gainesville, FL 32611}
	\email{haroldpolo@ufl.edu}
	\date{\today}
	
	\begin{abstract}
	 In this paper, we study some of the factorization aspects of rational multicyclic monoids, that is, additive submonoids of the nonnegative rational numbers generated by multiple geometric sequences. In particular, we provide a complete description of the rational multicyclic monoids $M$ that are hereditarily atomic (i.e., every submonoid of $M$ is atomic). Additionally, we show that the sets of lengths of certain rational multicyclic monoids are finite unions of multidimensional arithmetic progressions, while their unions satisfy the Structure Theorem for Unions of Sets of Lengths. Finally, we realize arithmetic progressions as the sets of distances of some additive submonoids of the nonnegative rational numbers.
	\end{abstract}
	
	\maketitle
	
	\section{Introduction} \label{sec:intro}
	The ring of integers $\mathcal{O}_K$ of an algebraic number field $K$ is not, in general, a unique factorization domain. However, we can still represent a nonzero nonunit element of $\mathcal{O}_K$ as a product of irreducibles. This poses the question of whether we can measure, through algebraic invariants, how far is an integral domain from being a unique factorization domain; this query is one of the driving forces behind factorization theory. Nowadays, factorization theory has branched out into several subfields of algebra, and many authors study the factorization aspects of objects such as matrices~\cite{DBNBJG2014}, modules~\cite{NBAG2014}, and orders~\cite{Smertnig2013}. 
	
	The systematic study of the factorization properties of additive submonoids of the nonnegative cone of $\qq$, also known as Puiseux monoids, started just three years ago in~\cite{GOTTI16}. Nevertheless, Puiseux monoids have been used to provide crucial examples in the realms of commutative ring theory due to the intrinsic complexity of their atomicity. For instance, Grams utilized these monoids to refute Cohn's assumption that every atomic integral domain satisfies the ACCP~\cite{Grams}. More recently, Coykendall and Gotti~\cite{JCFG2019} used Puiseux monoids to partially answer a question proposed by Gilmer in the 1980s~\cite[page 189]{Gilmer}. 
	
	Sets of lengths are perhaps the most investigated factorization invariants; they have been studied in the context of Krull monoids~\cite{AGWSQZ}, C-monoids~\cite{AFWH}, finitely generated monoids~\cite{HK1997}, and submonoids of $\nn_0^d$~\cite{GOTTI2019} (see also~\cite{YFAGFKST, AGFGST} for some recent work). An exhaustive description of the sets of lengths of Puiseux monoids is a herculean task. For example, if $M$ is the monoid generated by the reciprocals of primes then $\mathsf{L}(1)$ is the set of prime numbers and $\mathsf{L}(2)$ is the set of Goldbach's numbers~\cite[Section 6]{GOTTI2019L}. Then, how does one go about studying the sets of lengths of Puiseux monoids? This question motivated a series of articles investigating the sets of lengths of Puiseux monoids generated by well-structured sets (see, for example,~\cite{ScGG2019,GoOn2017}). In~\cite[Theorem~3.3]{ScGG2019}, Chapman et al. proved that the sets of lengths of a rational cyclic monoid (i.e., a Puiseux monoid generated by the elements of a geometric sequence) are arithmetic progressions. We extend this result to monoids that are natural generalizations of rational cyclic monoids and that we call canonical rational multicyclic monoids.
	
	The outline of this paper is as follows. In Section~2, we establish the notation we will be using throughout this article. Next we provide, in Section~3, a complete description of the rational multicyclic monoids (i.e., additive submonoids of the nonnegative rational numbers generated by multiple geometric sequences) that are hereditarily atomic. In Section~4, we study the sets of lengths of canonical rational multicyclic monoids. Here we prove that the sets of lengths of a canonical rational multicyclic monoid are finite unions of multidimensional arithmetic progressions (see Definition~\ref{def: MAP}), and we use this result to realize arithmetic progressions as the sets of distances of certain Puiseux monoids. Additionally, we show that canonical rational multicyclic monoids satisfy the Structure Theorem for Unions of Sets of Lengths (as stated in \cite[Theorem~4.2]{GaGe09}).

	\section{Background} \label{sec:definition of GL}
	\subsection{Notation}
	Throughout this paper, we let $\nn$ and $\nn_0$ denote the set of positive and nonnegative integers, respectively, while the symbol $\mathbb{P}$ stands for the set of prime numbers. For nonnegative integers $m$ and $n$, let $\llbracket m, n \rrbracket$ be the set of integers between $m$ and $n$, i.e.,
	\[
		\llbracket m,n \rrbracket \coloneqq \{k \in\nn_0 \mid m \leq k \leq n\}.
	\]
	Given a subset $S$ of the rational numbers, we let $S_{\geq t}$ denote the set of nonnegative elements of $S$ that are greater than or equal to $t$. In the same spirit we define $S_{>t}$, $S_{\leq t}$ and $S_{<t}$. For a positive rational number $q$, the relatively prime positive integers $n$ and $d$ for which $q = n/d$ are denoted by $\mathsf{n}(q)$ and $\mathsf{d}(q)$, respectively. We say that a positive fraction $n/m$ is \emph{proper} if $n < m$; otherwise, we denote $n/m$ \emph{improper}. Given $L_1, \ldots, L_n \subseteq \zz$ we denote $L_1 + \cdots + L_n = \left\{ l_1 + \cdots + l_n \mid l_i \in L_i \right\}$.

	\begin{definition} \label{def: MAP}
		For $d \in\nn$ and $l \in\nn_0 \cup \{\infty\}$ we set $P_l(d) \coloneqq d\zz \cap [0, ld]$. If $r$ is a positive integer then a nonempty subset $S \subseteq \nn_0$ is called an \emph{$r$-dimensional arithmetic progression} (with differences $d_1, \ldots, d_r \in\nn$) if
		\[
			S = \min S + \sum_{i = 1}^{r} P_{l_i}(d_i)
		\]
		for some $l_1, \ldots, l_r \in\nn_0\cup\{\infty\}$. We say that $S$ is a \emph{multidimensional arithmetic progression} (or \emph{MAP}) if $S$ is an $r$-dimensional arithmetic progression for some positive integer $r$.
	\end{definition}

	\begin{definition} \label{def: AAP}
		For $d \in \nn$ and $N \in \nn_0$, we say that a subset $S \subseteq \zz$ is an \emph{almost arithmetic progression} (or \emph{AAP}) with difference $d$ and bound $N$ if
		\[
		S = y + (S' + S^{*} + S'') \subseteq y + d\zz,
		\] 
		where $y \in\zz$ and $S^{*}$ is a nonempty arithmetic progression with difference $d$ such that $\min S^* = 0$, $S' \subseteq [-N,-1]$, and $S'' \subseteq \sup S^* + [1,M]$ (we assume that $S'' = \emptyset$ provided that $S^*$ is infinite).
	\end{definition}

	\subsection{Puiseux monoids}
	We say that a monoid $M$ is \emph{reduced} if the only invertible element of $M$ is the identity. From now on we assume that all monoids here are commutative, cancellative, and reduced. Let $M$ be a monoid, which is written additively (so we call the identity  element \emph{zero}). A nonzero element $x \in M$ is an \emph{atom} with the condition that $x$ cannot be expressed as the sum of two nonzero elements of $M$. We let $\mathcal{A}(M)$ represent the set of atoms of $M$. Now for a subset $S \subseteq M$, we denote by $\langle S \rangle$ the minimal submonoid of $M$ including $S$, and if $M = \langle S \rangle$ then we say that $S$ is a \emph{generating set} of $M$. The monoid $M$ is \emph{atomic} provided that $M = \langle\mathcal{A}(M)\rangle$. In addition, $M$ is called \emph{hereditarily atomic} if every submonoid of $M$ is atomic. On the other hand, we say that $x$ \emph{divides} y in $M$, denoted by $x \,|_M \,y$, if there exists $x' \in M$ such that $y = x + x'$ with $x,y \in M$. As usual, we use the notation $x \,|\, y$ to indicate that $x,y \in\zz$ and $x \,|_{(\zz,\times)}\, y$ with $x \neq 0$. 
	
	A subset $I$ of $M$ is an \emph{ideal} of $M$ provided that $I + M \subseteq I$; we say that $I$ is \emph{principal} if $I = x + M$ for some $x \in M$. The monoid $M$ satisfies the \emph{ascending chain condition on principal ideals} (or \emph{ACCP}) if there is no strictly increasing sequence of principal ideals of $M$. A monoid that satisfies the ACCP is atomic~(\cite[Proposition~1.1.4]{AGHK06}).
	
	\begin{definition} 
		A \emph{Puiseux monoid} is an additive submonoid of $\qq_{\geq 0}$.
	\end{definition}
	
	Puiseux monoids have a fascinating atomic structure. While some Puiseux monoids have no atoms at all~(\cite[Example~3.3]{GOTTI16}), others have exactly $m$ atoms for each positive integer $m$~(\cite[Proposition~5.4]{GOTTI16}). The atomicity of these monoids has received considerable attention lately (see~\cite{CGG2019} and references therein). The most studied family of Puiseux monoids is that one consisting of all rational cyclic monoids. 
	
	\begin{definition}
		The \emph{rational cyclic monoid}\footnote{Although these monoids are semirings (and were named `rational cyclic semirings' in~\cite{ScGG2019}), the operation of multiplication plays no role here.} over $r \in \qq_{>0}$ is the monoid generated by the nonnegative powers of $r$, i.e., $M_r \coloneqq \langle r^n \mid n\in\nn_0\rangle$. 
	\end{definition}
	
	\noindent To check whether a rational cyclic monoid is atomic is straightforward. Consider the following theorem.
		
	\begin{theorem}\cite[Theorem 6.2]{GG2018} \label{theorem: cyclic rational semirings are atomic unless generated by unit fraction}
		Let $r \in \qq_{>0}$ and consider the rational cyclic monoid $M_r$. The following statements hold:
		\begin{enumerate}
			\item if $\mathsf{d}(r) = 1$ then $M_r$ is atomic with $\mathcal{A}(M_r) = \{1\}$;
			\item if $\mathsf{d}(r) > 1$ and $\mathsf{n}(r) = 1$ then $M_r$ is not atomic with $\mathcal{A}(M_r) = \emptyset$;
			\item if $\mathsf{d}(r) > 1$ and $\mathsf{n}(r) > 1$ then $M_r$ is atomic with $\mathcal{A}(M_r) = \{r^n \mid n \in \nn_0\}$.
		\end{enumerate} 
	\end{theorem}

	\subsection{Sets of lengths and their unions}
	
	The \emph{factorization monoid} of $M$, denoted by $\mathsf{Z}(M)$, is the free commutative monoid on $\mathcal{A}(M)$. The elements of $\mathsf{Z}(M)$ are called \emph{factorizations}, and if $z = a_1 + \cdots + a_n$ is an element of $\mathsf{Z}(M)$ for $a_1, \ldots, a_n \in\mathcal{A}(M)$ then we say that $|z|$, the \emph{length} of $z$, is $n$. The unique monoid homomorphism $\pi\colon\mathsf{Z}(M) \to M$ satisfying that $\pi(a) = a$ for all $a \in\mathcal{A}(M)$ is called the \emph{factorization homomorphism} of $M$. For all $x \in M$, there are two important sets associated to $x$\,:
	\[
		\mathsf{Z}_M(x) \coloneqq \pi^{-1}(x) \subseteq \mathsf{Z}(M) \hspace{0.6 cm}\text{ and } \hspace{0.6 cm}\mathsf{L}_M(x) \coloneqq \{|z| : z \in\mathsf{Z}_M(x)\},
	\]
	which are called the \emph{set of factorizations} of $x$ and the \emph{set of lengths} of $x$, respectively; we omit subscripts when there is no risk of ambiguity. In addition, the collection $\mathcal{L}(M) \coloneqq \{\mathsf{L}(x) \mid x \in M\}$ is called the \emph{system of sets of lengths} of $M$. See~\cite{AG2016} for a survey about sets of lengths and the role they play in factorization theory. For the rest of this section, we assume that $M$ is an atomic monoid. We say that $M$ satisfies the \emph{finite factorization property} provided that $\mathsf{Z}(x)$ is finite for all $x \in M$. In this case we call $M$ an \emph{FF-monoid}. Similarly, we say that $M$ satisfies the \emph{bounded factorization property} provided that $\mathsf{L}(x)$ is finite for all $x \in M$. In this case we call $M$ a \emph{BF-monoid}. It is well known that BF-monoids satisfy the ACCP (see~\cite[Corollary~1.3.3]{AGHK06}).
	
	We now introduce unions of sets of lengths and local elasticities. For a positive integer $n$, we denote by $\,\mathcal{U}_n(M)$ the set of positive integers $m$ for which there exist $a_1, \ldots, a_n, a'_1, \ldots, a'_m \in \mathcal{A}(M)$ such that $a _1 + \cdots + a_n = a'_1 + \cdots + a'_m$. We say that $\,\mathcal{U}_n(M)$ is the \emph{union of sets of lengths} of $M$ containing $n$. We also say that $\rho_n(M) \coloneqq \sup \mathcal{U}_n(M)$ is the \emph{nth local elasticity} of $M$. Unions of sets of lengths were introduced in~\cite{ChWWS1990} and further studied in~\cite{MFAG2008,GaGe09,ST2019}.
	
	A factorization invariant that is closely related to sets of lengths is the \emph{set of distances} or \emph{delta set}. For a nonzero element $x \in M$ we say that $d$ is a \emph{distance} of $x$ provided that $\mathsf{L}(x) \cap [l, l + d] = \{l, l + d\}$ for some $l \in \mathsf{L}(x)$. The \emph{set of distances of $x$}, denoted by $\Delta(x)$, is the set consisting of all the distances of $x$. In addition, the set	
	\[	
	\Delta(M) \coloneqq \bigcup_{x \in M, \, x \neq 0} \Delta(x)
	\]   
	is called the \emph{set of distances} of $M$. In general, determining the set of distances of a given monoid is no simple task, and just some specific calculations are known (see, for example,~\cite{BCKR06, CKDH2010,CK2016,GSLL2017,AGWS2017,GeQZ2019}).

	\section{Rational Multicyclic Monoids}
	
	In this section we initiate the study of the atomic properties of Puiseux monoids generated by multiple geometric sequences, but first we make a definition to avoid long descriptions.
	
	\begin{definition} \label{def: rational multicyclic monoid}
		Let $\mathcal{B}$ be a finite subset of $\qq_{>0}$ and set $M_{\mathcal{B}} \coloneqq \langle b^n \mid b \in \mathcal{B},\, n \in \mathbb{N}_0 \rangle$. We say that $M_{\mathcal{B}}$ is the \emph{rational multicyclic monoid over $\mathcal{B}$} provided that $\mathcal{B}$ is minimal, that is, if $\mathcal{B}' \subsetneq \mathcal{B}$ then $M_{\mathcal{B}'} \subsetneq M_{\mathcal{B}}$. 
	\end{definition}
			
	Given a finite subset $\mathcal{B}$ of the positive rational numbers, $M_{\mathcal{B}}$ is the rational multicyclic monoid over some subset $\mathcal{B'}$ of $\mathcal{B}$; we call the elements of $\mathcal{B'}$ \emph{primitive generators}. Clearly, $\nn_0$ is a rational multicyclic monoid. Furthermore, $\nn_0$ is a submonoid of $M$ for all rational multicyclic monoids $M$ with at least one primitive generator. Especial cases of rational multicyclic monoids have been studied before.
	
	\begin{example}
		Let $b \in\qq_{>0}$ such that $\mathsf{n}(b),\! \mathsf{d}(b) > 1$ and consider the rational multicyclic monoid $M_{\mathcal{B}}$ with $\mathcal{B} = \{b, b^{-1}\}$. In~\cite[Proposition 3.5]{GOTTI2018}, it was proved not only that $M_{\mathcal{B}}$ is atomic but also that $\text{Aut}(M_{\mathcal{B}}) \cong \zz$.
	\end{example}
	
	\begin{example}
		For a fixed $r \in\qq_{>0}$ consider the rational cyclic monoid $\langle r^n \mid n\in \nn_0\rangle$. These monoids were introduced by Gotti and Gotti in~\cite{GG2018} and deeper studied by Chapman et al. in~\cite{ScGG2019}. It is well known that rational cyclic monoids are atomic unless $r$ is a unit fraction, i.e., $r = n^{-1}$ for some $n \in \nn_{>1}$ (Theorem~\ref{theorem: cyclic rational semirings are atomic unless generated by unit fraction}). 
	\end{example}
	
	As we mentioned earlier, rational multicyclic monoids are a generalization of rational cyclic monoids. However, we have not shown yet that the class of rational multicyclic monoids properly includes that one comprising rational cyclic monoids. Before doing so (Example~\ref{ex: rational multicyclic monoid that is not a semiring} below) we will establish a few facts about atomic rational multicyclic monoids. 
	
	Note that if a rational multicyclic monoid $M_{\mathcal{B}}$ has a unit fraction $1/d$ as primitive generator then $M_{\mathcal{B}}$ is not atomic. Indeed, since $M_{\mathcal{B}}$ is minimally generated (in the sense of Definition~\ref{def: rational multicyclic monoid}) by some subset $\mathcal{B'} \subseteq \mathcal{B}$ containing $1/d$, not all elements of the form $1/d^n$ with $n \in \nn_0$ are generated (as elements of a monoid) by the nonnegative powers of the elements of $\mathcal{B'}\setminus\{1/d\}$. This, along with the fact that $(1/d)^n = d(1/d)^{n + 1}$ for all $n \in \nn_0$, implies that $M_{\mathcal{B}}$ is not atomic. For the rest of the paper we tacitly assume that primitive generators of rational multicyclic monoids are not unit fractions. 
	
	Our next goal is to prove that when checking whether a rational multicyclic monoid $M_{\mathcal{B}}$ is atomic there is no loss in assuming that the elements of $\mathcal{B}$ are proper fractions. 
	
	\begin{proposition} \label{prop: taking off well-ordered sets of values greater than 1 does not change the atomicity of a MC}
		Let $M_{\mathcal{B}}$ be a rational multicyclic monoid and let $\mathcal{B'} = \mathcal{B}_{<1}$. Then $M_{\mathcal{B}}$ is atomic if and only if $M_{\mathcal{B'}}$ is atomic.
	\end{proposition} 
	
	\begin{proof}
		If $\mathcal{B} = \emptyset$ then our argument follows readily. Moreover, we can assume that $\mathcal{B}' \neq\emptyset$ by \cite[Proposition 4.5]{GOTTI19}. Since $M_{\mathcal{B}}$ is reduced, $\mathcal{A}(M_{\mathcal{B}'}) \supseteq \mathcal{A}(M_{\mathcal{B}}) \cap M_{\mathcal{B}'}$. This implies that $\mathcal{A}(M_{\mathcal{B}'}) = \mathcal{A}(M_{\mathcal{B}})_{\leq 1}$. Now suppose that $M_{\mathcal{B}}$ is atomic. For $b \in \mathcal{B}'$ and $n \in \nn_0$ we can write $b^n$ as the sum of elements of the set $\mathcal{A}(M_{\mathcal{B}})_{\leq 1}$. Hence $b^n\in\langle\mathcal{A}(M_{\mathcal{B}'})\rangle$ which, in turn, implies that $M_{\mathcal{B'}}$ is atomic.
		
		Conversely, suppose that $M_{\mathcal{B}'}$ is atomic. Let $y_0$ be a nonzero element of $M_{\mathcal{B}}$ and suppose by contradiction that $y_0\not\in\langle\mathcal{A}(M_{\mathcal{B}})\rangle$. Consequently, $y_0 = x_1 + x_2$ with $x_1, x_2$ nonzero elements of $M_{\mathcal{B}}$. Assume without loss that $x_1 \not\in\langle\mathcal{A}(M_{\mathcal{B}})\rangle$. Since the inclusion $\mathcal{A}(M_{\mathcal{B'}}) \subseteq \mathcal{A}(M_{\mathcal{B}})$ holds, we have that $x_1\not\in\langle\mathcal{A}(M_{\mathcal{B}'})\rangle$. Because the submonoid $M_{\mathcal{B'}}$ is atomic, there exists $y_1 \in M_{\mathcal{B}\setminus\mathcal{B}'}$ satisfying that $y_1 \,|_{M_{\mathcal{B}}} \,x_1$ and $y_1 \not\in \langle\mathcal{A}(M_{\mathcal{B}})\rangle$. Note that $y_0 > y_1$. Repeating the same reasoning for $y_1$, which is not an element of $\langle\mathcal{A}(M_{\mathcal{B}})\rangle$ (as it was the case for $y_0$), we obtain an element $y_2 \in M_{\mathcal{B}\setminus\mathcal{B}'}$ such that $y_2\not\in\langle\mathcal{A}(M_{\mathcal{B}})\rangle$ and $y_1 > y_2$. Using an inductive argument, it is not hard to show that there exists a strictly decreasing sequence $y_0 > y_1 > y_2 > \cdots$ of elements of $M_{\mathcal{B}\setminus\mathcal{B'}}$, but this contradicts \cite[Theorem~3.9]{GG2018}. Hence $y_0 \in\langle \mathcal{A}(M_{\mathcal{B}}) \rangle$, which concludes our proof.
	\end{proof}

	The following proposition plays a key role in this manuscript as it provides several examples of atomic rational multicyclic monoids with conspicuous sets of atoms. Consequently, we will turn to this proposition to construct a rational multicyclic monoid that is not (unlike rational cyclic monoids) a semiring, to realize arithmetic progressions as the sets of distances of certain Puiseux monoids, and to distinguish a family of rational multicyclic monoids whose sets of lengths are well structured.
	
	\begin{proposition} \label{prop: sufficient condition for being an atom}
		Set $M \coloneqq \langle b^n \mid b \in\mathcal{B},\, n \in\nn_0 \rangle$ with $\mathcal{B}$ a (not necessarily finite) subset of $\qq_{>0}\setminus\nn$ satisfying that $\mathsf{n}(b) \neq 1$ for all $b \in \mathcal{B}$. If $b_0$ is an element of $\mathcal{B}$ such that $\gcd(\mathsf{d}(b_0), \mathsf{d}(b)) = 1$ for each $b \in \mathcal{B}\setminus \{b_0\}$ then $b_0^n$ is an atom of $M$ for all $n \in \nn$. 
	\end{proposition}
	
	\begin{proof}
		Suppose, by contradiction, that $b_0^n$ is not an atom of $M$ for some $n \in \nn$. Thus,
		\begin{equation}\label{eq: atom a^k}
		b_0^n = c_1b_0^{m + n} + \cdots + c_{m + n}b_0 + c_{m + n + 1}b_1^{e_1} + \cdots + c_{m + n + k}b_k^{e_k}
		\end{equation}
		for coefficients $c_1, \ldots, c_{m + n + k} \in \nn_0$, exponents $e_1, \ldots, e_k \in \nn$ and elements $b_1, \ldots, b_k \in \mathcal{B}\setminus\{b_0\}$. Given that $\mathsf{d}(b_0)$ and $\mathsf{d}(b)$ are relatively prime numbers for all $b \in\mathcal{B}$ with $b \neq b_0$, the inequality $c_1 + \cdots + c_{m + n} > 0$ holds. First, we analyze the case where $b_0 < 1$. Under this assumption we have $c_{m + 1} = \cdots = c_{m + n} = 0$. So we can assume that $m \geq 1$ and $c_1 > 0$. In virtue of \cite[Lemma~3.1]{ScGG2019}, there is no loss in assuming that $c_l < \mathsf{d}(b_0)$ for each $l \in \llbracket 1, m + n \rrbracket$. After multiplying Equation~\ref{eq: atom a^k} by $N \coloneqq \mathsf{d}(b_1)^{e_1} \cdots \mathsf{d}(b_k)^{e_k}$ it is easy to see that $b_0$ is a rational root of the polynomial $c_1N x^{m + n} + \cdots - Nx^n + K$ for some $K \in \nn_0$. Then $\mathsf{d}(b_0) \mid c_1$ by the Rational Root Theorem stating that if $q \in \qq$ is a root of a polynomial $p$ in one variable with integer coefficients then $\mathsf{n}(q)$ and $\mathsf{d}(q)$ divide the constant term of $p$ and the leading coefficient of $p$, respectively. This contradiction concludes the proof for the case where $b_0 < 1$. We proceed in a similar fashion for the case $b_0 > 1$: in Equation~\ref{eq: atom a^k} we have $c_ 1 = \cdots = c_{m + 1} = 0$, which implies that $b_0$ is a rational root of the polynomial $-Nx^n + \cdots + K$. Again, applying the Rational Root Theorem we obtain that $\mathsf{d}(b_0) \mid N$, which is a contradiction. Therefore, our result follows.
	\end{proof}
	
	The examples of rational multicyclic monoids that we have seen so far are multiplicatively closed. However, this is not always the case as the next example illustrates. 
	
	\begin{example} \label{ex: rational multicyclic monoid that is not a semiring}
		Let $p_1$ and $p_2$ be two prime numbers such that $2 < p_1 < p_2$. Consider the rational multicyclic monoid $M_{\mathcal{B}}$ with $\mathcal{B} = \{p_1/p_2,\, (p_2 - p_1)/p_1\}$. Note that if a rational multicyclic monoid $M$ is multiplicatively closed then $\mathcal{A}(M) = \emptyset$ if and only if $1 \not\in\mathcal{A}(M)$. We have
		\[
			1 = \frac{p_1}{p_2} \left(\frac{p_2 - p_1}{p_1}\right) + \frac{p_1}{p_2}.
		\]
		This, along with the fact that $\mathcal{A}( M_{\mathcal{B}}) \neq \emptyset$ by Proposition~\ref{prop: sufficient condition for being an atom}, implies that $M_{\mathcal{B}}$ is not multiplicatively closed.
	\end{example}

\noindent Example~\ref{ex: rational multicyclic monoid that is not a semiring} shows that the family of rational multicyclic monoids properly includes that one comprising rational cyclic monoids.

If we fix $n \in \nn$ then we can find infinitely many atomic rational multicyclic monoids with exactly $n$ proper fractions as primitive generators by Proposition~\ref{prop: sufficient condition for being an atom}. Next, we prove an equivalent result for non-atomic rational multicyclic monoids, which gives evidence of the complexity of classifying atomic rational multicyclic monoids.

%

\begin{theorem} \label{theorem: atomic description of trivial GMCs}
	Let $M_{\mathcal{B}}$ be a rational multicyclic monoid and $n$ a nonnegative integer. The following statements hold:
		\begin{enumerate}
			\item $M_{\mathcal{B}}$ is hereditarily atomic if and only if $\mathsf{n}(b) \geq \mathsf{d}(b)$ for all $b \in \mathcal{B}$;\vspace{3 pt}
			\item there are infinitely many non-atomic and non-isomorphic rational multicyclic monoids with exactly $n$ proper fractions as primitive generators if and only if $n \geq 2$.
		\end{enumerate} 
	\end{theorem}
		
	\begin{proof}
		For the reverse implication of $(1)$ it is not hard to see that $0$ is not a limit point of the nonzero elements of $M_{\mathcal{B}}$, which implies that $M_{\mathcal{B}}$ is a BF-monoid by \cite[Proposition~4.5]{GOTTI19}. Hence $M_{\mathcal{B}}$ satisfies the ACCP, and the reverse implication follows. As for the direct implication, we know that $M_{\mathcal{B}}$ is hereditarily atomic and, \emph{a fortiori}, atomic. If $\mathcal{B} = \emptyset$ then our argument follows immediately. Now assume, by way of contradiction, that $\mathcal{B}$ contains a proper fraction. Then there exists $q \in \mathcal{A}(M_{\mathcal{B}})$ such that $1 < \mathsf{n}(q) < \mathsf{d}(q)/2$. For each $n \in \nn$, we have
		\begin{equation} \label{eq: equation inducing ACCP property}
		\mathsf{d}(q)q^n \!= \mathsf{d}(q)q^{n + 1}\!+ (\mathsf{d}(q) - \mathsf{n}(q))q^n\!
		\end{equation}
		as the reader can check. Let $M' = \langle \mathsf{d}(q)q^n, \,(\mathsf{d}(q) - \mathsf{n}(q))q^n \mid n \in \nn \rangle$. Clearly, $M'$ is a submonoid of $M_{\mathcal{B}}$, which implies that $M'$ is atomic. In virtue of Equation~\ref{eq: equation inducing ACCP property}, $\mathsf{d}(q)q^n \not\in\mathcal{A}(M')$ for any $n \in\nn$. Consequently, $\mathcal{A}(M') \subseteq \{(\mathsf{d}(q) - \mathsf{n}(q))q^n \mid n \in \nn\}$. Fix $n \in \nn_{\geq 2}$. As $M'$ is atomic, we have
		\begin{equation} \label{eq: original equation}
		\mathsf{d}(q)q^n = \sum_{i = 1}^{k} c_i (\mathsf{d}(q) - \mathsf{n}(q))q^{m_i}
		\end{equation} 	
		for some index $k \in\nn_{>1}$, coefficients $c_1, \ldots, c_k \in \nn$ and exponents $m_1, \ldots, m_k \in \nn$. Without loss of generality, we can assume that $m_1 < \cdots < m_k$. Since the inequality $2\cdot\mathsf{n}(q) < \mathsf{d}(q)$ holds, it is not hard to check that $\mathsf{d}(q)q^n < (\mathsf{d}(q) - \mathsf{n}(q))q^{n - 2}$, which implies that $m_1 \geq n - 1$. After multiplying both sides of Equation~\ref{eq: original equation} by $\mathsf{d}(q)^{m_k}$ we obtain
		\begin{equation} \label{eq: original equation 2}
		\mathsf{n}(q)^n\mathsf{d}(q)^{m_k - n + 1} = \sum_{i = 1}^{k} c_i (\mathsf{d}(q) - \mathsf{n}(q))\mathsf{n}(q)^{m_i}\mathsf{d}(q)^{m_k - m_i}\!,
		\end{equation}
		where both sides of Equation~\ref{eq: original equation 2} represent integers since $n - 1 \leq m_1 < \cdots < m_k$. Now take $p \in\mathbb{P}$ such that $p \,| \,\mathsf{d}(q) - \mathsf{n}(q)$. Note that such a prime $p$ must exist given that the inequalities $1 < \mathsf{n}(q) < \mathsf{d}(q)/2$ hold. Since Equation~\ref{eq: original equation 2} also holds, either $p \,| \,\mathsf{n}(q)$ or $p \,| \,\mathsf{d}(q)$. This contradiction proves that our hypothesis is untenable. Therefore, $\mathcal{B}$ contains no proper fraction, and $(1)$ follows.
		
		The direct implication of $(2)$ follows after $(1)$, Proposition~\ref{prop: taking off well-ordered sets of values greater than 1 does not change the atomicity of a MC}, and Theorem~\ref{theorem: cyclic rational semirings are atomic unless generated by unit fraction} (recall we assume that primitive generators of rational multicyclic monoids are not unit fractions). As for the reverse implication, let $p_0, \ldots, p_n$ be a finite sequence of prime numbers such that $p_0 < p_1 < p_0\cdot p_1 + 1 < p_2 < \cdots < p_n$; it is easy to see that there are infinitely many such sequences. Let 
		\[
			\mathcal{B} = \left\{\frac{p_0}{p_2}, \frac{p_1}{p_2}, \frac{p_0p_1}{p_3}, \ldots, \frac{p_0p_1}{p_n}\right\}
		\]
	if $n \geq 3$; otherwise, let $\mathcal{B} = \{p_0/p_2, p_1/p_2\}$. Now consider the rational multicyclic monoid $M_{\mathcal{B}}$. We shall prove that $\mathcal{B}$ minimally generates $M_{\mathcal{B}}$ (in the sense of Definition~\ref{def: rational multicyclic monoid}). Note that if
	\[
		M' = \left\langle b^n \mid b \in \mathcal{B}\setminus\{p_0/p_2\},\, n \in \nn_0 \right\rangle
	\]
	then $p_0/p_2 \not\in M'$ since the numerators of the elements of $\mathcal{B}\setminus \{p_0/p_2\}$ are divisible by $p_1$. This means that if a subset $\mathcal{B'} \subseteq \mathcal{B}$ does not contain $p_0/p_2$ then $M_{\mathcal{B'}}$ is a proper submonoid of $M_{\mathcal{B}}$. The same reasoning applies, \emph{mutatis mutandis}, to the element $p_1/p_2$. As for the rest of the elements of $\mathcal{B}$ (if there is any), they are atoms of $M_{\mathcal{B}}$ by Proposition~\ref{prop: sufficient condition for being an atom}. Hence if $\mathcal{B'}$ is a proper subset of $\mathcal{B}$ then $M_{\mathcal{B'}}$ is a proper submonoid of $M_{\mathcal{B}}$. In other words, $M_{\mathcal{B}}$ is the rational multicyclic monoid over $\mathcal{B}$. Moreover, $(p_0/p_2)^m \not\in\mathcal{A}(M_{\mathcal{B}})$ for any $m \in\nn$. Indeed, for a fixed $m \in\nn$, there exists $N \in\nn_{>m}$ such that the inequality $(p_2/p_0)^m < (p_2/p_0p_1)^N$ holds. Thus, 
		\[
		p_0^Np_1^N - p_0^N - p_1^N < (p_0p_1)^N < p_0^mp_2^{N - m}\!.
		\]
		Then there exist $\alpha, \beta \in \nn_0$ such that $\alpha p_0^N + \beta p_1^N = p_0^mp_2^{N - m}$ which implies that $\alpha, \!\beta \neq 0$ and $\alpha(p_0/p_2)^N + \beta(p_1/p_2)^N = (p_0/p_2)^m$.  Hence $M_{\mathcal{B}}$ is not atomic, and the proof follows after~\cite[Proposition 3.2]{GOTTI2018}. 
	\end{proof}
	
	\begin{cor} \label{cor: characterization of multicyclic monoid satisfying the ACCP condition}
		Let $M_{\mathcal{B}}$ be a rational multicyclic monoid. If $M_{\mathcal{B}}$ satisfies the ACCP then $\mathsf{n}(b) \geq \mathsf{d}(b)$ for all $b \in \mathcal{B}$.
	\end{cor}
	\begin{remark}
		Corollary~\ref{cor: characterization of multicyclic monoid satisfying the ACCP condition} was first proved in~\cite[Corollary 4.4]{CGG2019}. 
	\end{remark}

\section{Sets of Lengths of Canonical Rational Multicyclic Monoids}

In this section we introduce an atomic family of rational multicyclic monoids and characterize their sets of lengths and unions of sets of lengths. Moreover, we extend \cite[Theorem~3.3]{ScGG2019} stating that the sets of lengths of a rational cyclic monoid are arithmetic progressions to bigger families of rational multicyclic monoids. We conclude by realizing arithmetic progressions as the sets of distances of certain Puiseux monoids. 

\begin{definition}
	Let $M_{\mathcal{B}}$ be a rational multicyclic monoid. We say that $M_{\mathcal{B}}$ is \emph{canonical} provided that $\mathcal{B} \cap \nn = \emptyset$ and $\gcd(\mathsf{d}(b), \mathsf{d}(b')) = 1$ for all $b,b' \in \mathcal{B}$ with $b \neq b'$\!.
\end{definition}

\begin{remark} \label{remark: atoms of a canonical rational multicyclic monoid}
	Because of Proposition~\ref{prop: sufficient condition for being an atom}, a canonical rational multicyclic monoid $M_{\mathcal{B}}$ is atomic and $\{b^n \mid b \in\mathcal{B},\, n \in\nn\} \subseteq \mathcal{A}(M_{\mathcal{B}})$. In fact, it is not hard to check that $ \mathcal{A}(M_{\mathcal{B}}) = \{b^n \mid b \in\mathcal{B},\, n \in\nn_0\}$, which implies that $\mathcal{B}$ is minimal (in the sense of Definition~\ref{def: rational multicyclic monoid}).
\end{remark}

Before proving our main result (Theorem~\ref{theorem: characterization of sets of lengths of canonical multicyclic monoids}), we need to collect some technical lemmas. For the rest of the section, given a summation $\sum_{i = 0}^{n} c_i b_i^{e_i}$ we assume without loss that $e_i = 0$ if and only if $i = 0$ and if $b_i = b_j$ then $e_i < e_j$ if and only if $i < j$.

\begin{definition} \label{def: hub factorization}
	Let $x$ be a nonzero element of a canonical rational multicyclic monoid $M_{\mathcal{B}}$. Given a factorization $z = \sum_{i = 0}^{n} c_i b_i^{e_i} \in \mathsf{Z}(x)$ with $n, c_i,e_i \in\nn_0$ and $b_i \in \mathcal{B}$ for every $i \in \llbracket 0,n \rrbracket$, we say that $z$ is a \emph{hub factorization} of $x$ if $c_i < \mathsf{d}(b_i)$ for all $i \in \llbracket 1,n \rrbracket$.  
\end{definition}

\begin{lemma}\label{lemma: hub factorization unique iff M is canonical}
	Let $M_\mathcal{B}$ be a canonical rational multicyclic monoid. Then every nonzero element of $M_{\mathcal{B}}$ has exactly one hub factorization.
\end{lemma}

\begin{proof}
	Let $x$ be a nonzero element of $M_{\mathcal{B}}$. First, we prove that there exists a hub factorization of $x$. To do so, we describe an algorithm to transform a given factorization $z \in \mathsf{Z}(x)$ into a hub factorization of $x$. Let $z = \sum_{i = 0}^{n} c_i b_i^{e_i} \in \mathsf{Z}(x)$ with $n, c_i,e_i \in\nn_0$ and $b_i \in \mathcal{B}$ for every $i \in \llbracket 0,n \rrbracket$. If the inequality $c_i \geq \mathsf{d}(b_i)$ holds for some $i \in \llbracket 1,n \rrbracket$ then $c_i = q_i \mathsf{d}(b_i) + r_i$ with $q_i \in\nn$ and $r_i \in \llbracket 0,\mathsf{d}(b_i) - 1 \rrbracket$. Then we can modify $c_ib_i^{e_i}$ as follows
	\[
		c_i b_i^{e_i} = [q_i \mathsf{d}(b_i) + r_i]b_i^{e_i} = r_ib_i^{e_i} + q_i \mathsf{d}(b_i)b_i^{e_i} = r_ib_i^{e_i} + q_i \mathsf{n}(b_i)b_i^{e_i - 1}\!.
	\]
	This modification reduces the exponent of the summand $cb^e$ violating the condition $c < \mathsf{d}(b)$, where $e > 0$, which means that we cannot carry out this transformation infinitely many times. Repeating this reasoning for the summands $c_jb_j^{e_j}$ of $z$ for which $c_j \geq \mathsf{d}(b_j)$, we obtain a hub factorization of $x$. 
	
	Now let $z_h = \sum_{i = 0}^{n} c_i b_i^{e_i}$ and $z_h' = \sum_{i = 0}^{n} d_i b_i^{e_i}$ be two hub factorizations of $x$ with $n,c_i,e_i,d_i\in\nn_0$ and $b_i \in \mathcal{B}$ for every $i \in\llbracket 0,n \rrbracket$; there is no loss in assuming that only the coefficients of the summations may be different. For the sake of a contradiction, suppose that $z_h \neq z_h'$. Then there exists $j\in \llbracket 1,n \rrbracket$ such that $c_j \neq d_j$. Assuming that $j$ is as large as possible, we have
	\begin{equation} \label{eq: unique hub factorization}
	(c_j - d_j)b_j^{e_j} = \sum_{i = 0}^{j - 1} (d_i - c_i) b_i^{e_i}.
	\end{equation}
	After clearing denominators in Equation~\ref{eq: unique hub factorization}, it is easy to see that $\mathsf{d}(b_j) \mid c_j - d_j$. But this is a contradiction as $c_j,\!d_j < \mathsf{d}(b_j)$. Therefore, $z_h = z_h'$.
\end{proof}

In the proof of Lemma~\ref{lemma: hub factorization unique iff M is canonical} we established that a factorization $z \in \mathsf{Z}(x)$ of a nonzero element $x$ can be transformed into the hub factorization of $x$ by a finite sequence of modifications. We record this observation in Lemma~\ref{lemma: any factorization of an element can be transformed into the hub factorization of that element} for future reference.

\begin{lemma} \label{lemma: any factorization of an element can be transformed into the hub factorization of that element}
	Let $M_{\mathcal{B}}$ be a canonical rational multicyclic monoid and $z \in \mathsf{Z}(x)$ a factorization of a nonzero element $x$. If $z_h$ is the hub factorization of $x$ then there exist factorizations $z = z_1, \ldots, z_n = z_h \in \mathsf{Z}(x)$ satisfying that $\big||z_i| - |z_{i + 1}|\big| = |\mathsf{n}(b) - \mathsf{d}(b)|$ for some $b \in \mathcal{B}$ and all $i \in\llbracket 1,n-1 \rrbracket$.
\end{lemma}

\begin{remark} \label{remark: we can transform the hub factorization into any other factorization}
	Let $z_h = \sum_{i = 0}^{n} c_i b_i^{e_i}$ be the hub factorization of a nonzero element $x$, where $n,c_i,e_i \in\nn_0$ and $b_i \in \mathcal{B}$ for every $i \in\llbracket 0,n \rrbracket$. By symmetry, we can transform $z_h$ into a given factorization $z \in \mathsf{Z}(x)$ by a finite sequence of modifications of the form $\mathsf{n}(b_i)b_i^m = \mathsf{d}(b_i)b_i^{m + 1}$ with $b_i \in \mathcal{B}$ and $m \in \nn_0$. Therefore, $z_h$ is the factorization of maximum length of $x$ provided that $c_i < \min\{\mathsf{n}(b_i), \mathsf{d}(b_i)\}$ for $i \in \llbracket 1,n \rrbracket$ and $c_0 < \mathsf{n}(b)$ for all $b \in \mathcal{B}_{<1}$.
\end{remark}

\begin{lemma} \label{lemma: factorization of minimum length when all primitive generators are proper fractions}
	Let $\mathcal{B}$ be a (not necessarily finite) subset of $\mathbb{Q}_{<1}$ satisfying that $\mathsf{n}(b) \neq 1$ and $\gcd(\mathsf{d}(b),\mathsf{d}(b')) = 1$ for all $b,b' \in\mathcal{B}$ with $b \neq b'$. Consider the atomic monoid $M = \langle b^n \mid b \in\mathcal{B},\, n \in\nn_0 \rangle$. If $x$ is a nonzero element of $M$ and $z = \sum_{i = 0}^{n} c_i b_i^{e_i}$ a factorization of $x$ with $n, c_i,e_i \in\nn_0$ and $b_i \in \mathcal{B}$ for every $i \in \llbracket 0,n \rrbracket$ then 
	the following statements hold:
	\begin{enumerate}
		\item $\min\mathsf{L}(x) = |z|$ if and only if $c_i < \mathsf{d}(b_i)$ for all $i \in\llbracket 1,n \rrbracket$;
		\item there exists exactly one factorization in $\mathsf{Z}(x)$ of minimum length.
		\item if $z$ is the factorization of minimum length of $x$ then, for $z' \in \mathsf{Z}(x)$, there exist factorizations $z' = z_1, \ldots, z_n = z$ such that $\big||z_i| - |z_{i + 1}|\big| = |\mathsf{n}(b) - \mathsf{d}(b)|$ for some $b \in \mathcal{B}$ and all $i \in\llbracket 1,n-1 \rrbracket$.
	\end{enumerate} 
\end{lemma}

\begin{proof}
	We leave the proofs of $(1)$ and $(2)$ to the reader as they mimick their counterparts for rational cyclic monoids (see \cite[Lemma~3.1]{ScGG2019}). As for $(3)$, let $z' = \sum_{i = 0}^{m} d_i b_i^{e_i} \in \mathsf{Z}(x)$ with $m, d_i,e_i \in\nn_0$ and $b_i \in \mathcal{B}$ for every $i \in \llbracket 0,m \rrbracket$. There is no loss in assuming that $m = n$. It is not hard to describe an algorithm to transform $z_1 \coloneqq z'$ into $z$.  If $z_1 \neq z$ then there exists $i \in \llbracket 1,n \rrbracket$ such that $d_i \geq \mathsf{d}(b_i)$ by~(1). Then applying the identity $\mathsf{d}(b_i)b_i^{e_i} = \mathsf{n}(b_i)b_i^{e_i - 1}$ we obtain a factorization $z_2 \in \mathsf{Z}(x)$ such that $|z_1| > |z_2|$. Repeating the same reasoning for $z_2$ we have that either $z_2 = z$ or there exists $z_3 \in \mathsf{Z}(x)$ such that $|z_2| > |z_3|$, and so on. Since there is no strictly decreasing sequence of positive integers, our procedure eventually stops, from which $(3)$ follows readily.
\end{proof}

\begin{remark} \label{remark: when all primitive generators are proper fractions there is no loss in assuming that the monoid is a canonical rational multicyclic monoid}
	With notation as in Lemma~\ref{lemma: factorization of minimum length when all primitive generators are proper fractions}, note that we used modifications of the form $\mathsf{d}(b_i)b_i^{e_i} = \mathsf{n}(b_i)b_i^{e_i - 1}$ with $i \in \nn$ to transform $z'$ into $z$, the factorization of minimum length of $x$. Consequently, if there exists $k \in \llbracket 1,n \rrbracket$ such that $c_k = 0$ and $d_k \neq 0$ then $\mathsf{n}(b_k)\, |_M \,x$.
\end{remark}


Now we are in a position to prove the main result of this paper.

\begin{theorem} \label{theorem: characterization of sets of lengths of canonical multicyclic monoids}
	Let $x$ be a nonzero element of a canonical rational multicyclic monoid $M_{\mathcal{B}}$. Then the following statements hold:\vspace{2 pt}
	\begin{enumerate}
		\item $\mathsf{L}(x)$ is the union of finitely many MAPs. Furthermore, $|\mathsf{L}(x)| = \infty$ if and only if $\mathsf{n}(b)b^l \mid_{M_{\mathcal{B}}} x$ for some $b \in\mathcal{B}_{< 1}$ and $l \in\nn_0$;\vspace{3pt}
		\item $\mathsf{L}(x)$ is an arithmetic progression if $|\mathsf{n}(b) - \mathsf{d}(b)| = |\mathsf{n}(b') - \mathsf{d}(b')|$ for all $b,b' \in \mathcal{B}$.
	\end{enumerate}
\end{theorem}

\begin{proof}
	Let $z_h = \sum_{i = 0}^{n} c_i b_i^{e_i}$ be the hub factorization of $x$, where $n,c_i, e_i \in\nn_0$ and $b_i \in\mathcal{B}$ for each $i \in\llbracket 0,n \rrbracket$ (Lemma~\ref{lemma: hub factorization unique iff M is canonical}), and set 
	\begin{equation*} \label{eq: terms with reduction}
	V \coloneqq \left\{b_i \in \mathcal{B} \mid c_i \geq \mathsf{n}(b_i), \,i \in \llbracket 1,n \rrbracket\right\}, \hspace{1 cm} \mathcal{U} \coloneqq \left\{U \subseteq \mathcal{B} \,\,\Bigg| \,\sum_{b \in U} \mathsf{n}(b) \leq c_0\right\},
	\end{equation*}
	and $\mathcal{W} \coloneqq \{V \cup U \mid U \in \mathcal{U}\}$. Note that $\emptyset\in \mathcal{U}$, which implies that $\mathcal{W} \neq \emptyset$. 
	
	To tackle the first statement of~$(1)$, we start by analyzing the case where the elements of $\mathcal{B}$ are proper fractions. Under this assumption $z_h$ is the factorization of minimum length of $x$ by Lemma~\ref{lemma: factorization of minimum length when all primitive generators are proper fractions}. For each $W = \{b_1, \ldots, b_m\} \in \mathcal{W}$, we set 
	\[
	\mathsf{L}_W \coloneqq \left\{|z_h| + \sum_{j = 1}^{m} P_{\infty}(\mathsf{d}(b_j) - \mathsf{n}(b_j))\right\};
	\]
	on the other hand, if $\emptyset\in \mathcal{W}$ then we set $L_{\emptyset} \coloneqq \{|z_h|\}$. We shall prove that the equation $\mathsf{L}(x) = \bigcup_{W \in \mathcal{W}} \mathsf{L}_W$ holds. 
	
	Consider a factorization $z = \sum_{i = 0}^{k} c'_i b_i^{e_i} \in \mathsf{Z}(x)$, where $k,c_i' \in\nn_0$ and $b_i \in\mathcal{B}$ for each $i \in\llbracket 0,k \rrbracket$. Without loss of generality we can assume that $k = n$. Next, we describe an algorithm to transform $z_1 \coloneqq z$ into $z_h$ (we already described a similar procedure in the proof of Lemma~\ref{lemma: hub factorization unique iff M is canonical}). Through our iterations, we are going to keep track of two subsets $V'$ and $U'$ of $\mathcal{B}$. Initially, we have $V' = U' = \emptyset$. The first step is to check whether the factorizations $z_1$ and $z_h$ coincide. If this is the case then we stop. On the other hand, if $z_1 \neq z_h$ then $z_1$ is not the factorization of minimum length of $x$, which implies that $c'_i \geq \mathsf{d}(b_i)$ for some $i \in \llbracket 1,n \rrbracket$ by Lemma~\ref{lemma: factorization of minimum length when all primitive generators are proper fractions}. Hence $c'_i = q_i \mathsf{d}(b_i) + r_i$ with $q_i \in\nn$ and $r_i \in \llbracket 0,\mathsf{d}(b_i) - 1 \rrbracket$. We modify $c'_ib_i^{e_i}$ as follows
	\begin{equation} \label{eq: transformation to obtain the hub factorization}
	c'_i b_i^{e_i} = [q_i \mathsf{d}(b_i) + r_i]b_i^{e_i} = r_ib_i^{e_i} + q_i \mathsf{d}(b_i)b_i^{e_i} = r_ib_i^{e_i} + q_i \mathsf{n}(b_i)b_i^{e_i - 1}\!,
	\end{equation}
	and as a result we obtain a factorization $z_2 \in \mathsf{Z}(x)$ such that $|z_1| > |z_2|$. We add $b_i$ to $V'$, and if $e_i = 1$ then we also add $b_i$ to $U'$. We repeat the first step over the factorization $z_2 \in \mathsf{Z}(x)$, and so on. Since there is no strictly decreasing sequence of positive integers, our procedure eventually stops. Then there exist factorizations $z = z_1, \ldots, z_t = z_h \in \mathsf{Z}(x)$ and (possibly repeated) elements $b_{k_1}, \ldots, b_{k_{t - 1}} \in \mathcal{B}$ such that $|z_j| - |z_{j + 1}| = \mathsf{d}(b_{k_j}) - \mathsf{n}(b_{k_j})$ for all $j \in \llbracket 1,t - 1 \rrbracket$. Thus,
	\[
		|z| = |z_h| + \sum_{j = 1}^{t - 1} \mathsf{d}(b_{k_j}) - \mathsf{n}(b_{k_j}).
	\]
	Moreover, it is not hard to see that due to the nature of the transformation~\ref{eq: transformation to obtain the hub factorization} the set $W' = V' \cup U'$ is either empty or an element of $\mathcal{W}$. If $W' = \emptyset$ then $z = z_h$; otherwise, we have $|z| \in L_{W'}$ for some $W' \in \mathcal{W}$. Either way, the inclusion $\mathsf{L}(x) \subseteq \bigcup_{W \in \mathcal{W}} \mathsf{L}_W$ holds.
	
	For the reverse inclusion, fix $W = (V \cup U) \in \mathcal{W}$ with $U \in \mathcal{U}$. We can assume without loss of generality that $W \neq \emptyset$. Note that we can write $z_h$ as 
	\begin{equation} \label{eq: Equation}
		z_h = c''_0 + \sum_{b_i \in U} \mathsf{n}(b_i) + \sum_{i = 1}^n c_i b_i^{e_i},	
	\end{equation}
	where $c''_0 \in\nn_0$. For each $b \in W$, there exists a summand $s_b = cb^e$ in the right-hand side of factorization~\ref{eq: Equation} such that $c \geq \mathsf{n}(b)$ (and $e \geq 0$). Consequently, by applying the identity $\mathsf{n}(b)b^{e} = \mathsf{d}(b)b^{e + 1}$ we can generate a factorization $z_1 \in \mathsf{Z}(x)$ such that $|z_1| = |z_h| + \mathsf{d}(b) - \mathsf{n}(b)$. Note that we can (re)apply the aforementioned identity as many times as we want given that $b < 1$, which means that for each $m_b \in\nn_0$ there exists a factorization $z^* \in \mathsf{Z}(x)$ such that $|z^*| = |z_h| + m_b(\mathsf{d}(b) - \mathsf{n}(b))$. Moreover, for distinct elements $b$ and $b'$ in $W$ we can carry out similar transformations on $s_b$ and $s_{b'}$ simultaneously. Then it is not hard to see that $L_W \subseteq \mathsf{L}(x)$ for each $W \in \mathcal{W}$. This, in turn, implies that $\bigcup_{W \in \mathcal{W}} \mathsf{L}_W \subseteq \mathsf{L}(x)$. 
	
	Now we proceed to prove the general case. Let $\mathcal{B}' = \mathcal{B}_{<1}$. Consider the canonical rational multicyclic monoids $M_{\mathcal{B}'}$ and $M_{\mathcal{B}\setminus\mathcal{B}'}$, which are clearly submonoids of $M_{\mathcal{B}}$. There is no loss in assuming that neither $M_{\mathcal{B}'}$ nor $M_{\mathcal{B}\setminus\mathcal{B}'}$ is the trivial Puiseux monoid $\{0\}$. Furthermore, $M_{\mathcal{B}\setminus\mathcal{B}'}$ is an FF-monoid by~\cite[Theorem 5.6]{GOTTI19}, and it is not hard to see that $\mathcal{L}(M_{\mathcal{B}'}) \cup \mathcal{L}(M_{\mathcal{B}\setminus\mathcal{B}'}) \subseteq \mathcal{L}(M_{\mathcal{B}})$ by Remark~\ref{remark: atoms of a canonical rational multicyclic monoid}. Now set
	\[
		\mathcal{D}(x) \coloneqq \left\{(y,y') \in  M_{\mathcal{B}\setminus \mathcal{B'}} \times M_{\mathcal{B'}} \mid x = y + y'\right\}.
	\]
	Note that $1 \leq |\mathcal{D}(x)| < \infty$ since there are only finitely many elements of $M_{\mathcal{B}\setminus\mathcal{B}'}$ dividing $x$ in $M_{\mathcal{B}}$. Thus,
	\begin{equation*}
		\begin{split}
			\mathsf{L}(x) = \mathsf{L}_{M_{\mathcal{B}}}(x) & = \bigcup_{(y,y') \in\mathcal{D}(x)} \mathsf{L}_{M_{\mathcal{B}\setminus\mathcal{B}'}}(y) + \mathsf{L}_{M_{\mathcal{B}'}}(y')\\
			& = \bigcup_{(y,y') \in\mathcal{D}(x)} \left(\bigcup_{l \in \mathfrak{L}(y)} l + \mathsf{L}_{M_{\mathcal{B}'}}(y')\right)\\
			& = \bigcup_{(y,y') \in\mathcal{D}(x)}\, \bigcup_{l \in \mathfrak{L}(y)} \left(l + \mathsf{L}_{M_{\mathcal{B}'}}(y')\right),
		\end{split}
	\end{equation*}
	where $\mathfrak{L}(y) = \mathsf{L}_{M_{\mathcal{B}\setminus\mathcal{B}'}}(y)$. We already established that $\mathsf{L}_{M_{\mathcal{B}'}}(y')$ is the union of finitely many MAPs for all $y' \in M_{\mathcal{B'}}$. This, along with the fact that $\mathfrak{L}(y)$ and $\mathcal{D}(x)$ are finite sets, implies that $\mathsf{L}(x)$ is the union of finitely many MAPs. Note that either $\mathsf{L}(x)$ is finite or $\mathsf{n}(b)b^l \mid_{M_{\mathcal{B}}} x$ for some $b \in\mathcal{B}_{< 1}$ and $l \in\nn_0$.   
	
	The direct implication of the second statement of $(1)$ follows readily after our previous observation. As for the reverse implication, note that $z = \mathsf{n}(b)b^l$ is the hub factorization of $y = \pi(z)$. From what we just proved, it follows that $|\mathsf{L}(y)| = \infty$ which, in turn, implies that $|\mathsf{L}(x)| = \infty$ given that $y \,|_{M_{\mathcal{B}}}\, x$. 
	
	Now we proceed to prove $(2)$. Let $z_0 \in \mathsf{Z}(x)$ be a factorization of minimum length of $x$ and let $l = |z_0|$. There is no loss in assuming that $\mathcal{B} \neq \emptyset$, so take $b \in\mathcal{B}$. Let $z \in \mathsf{Z}(x)$. In virtue of Lemma~\ref{lemma: any factorization of an element can be transformed into the hub factorization of that element}, there exist factorizations $z = z_1, \ldots, z_n = z_h \in \mathsf{Z}(x)$ such that $z_h$ is the hub factorization of $x$ and $\big||z_i| - |z_{i + 1}|\big| = |\mathsf{n}(b) - \mathsf{d}(b)|$ for all $i \in\llbracket 1,n-1 \rrbracket$. Similarly, there exist factorizations $z_0 = z'_1, \ldots, z'_m = z_h \in \mathsf{Z}(x)$ such that $\big||z'_i| - |z'_{i + 1}|\big| = |\mathsf{n}(b) - \mathsf{d}(b)|$ for all $i \in\llbracket 1,m-1 \rrbracket$. Since $z_0$ is a factorization of minimum length of $x$, we have $|z| = |z_0| + k |\mathsf{n}(b) - \mathsf{d}(b)|$ for some $k \in \nn_0$. Thus, 
	\begin{equation} \label{inclusion: set of lengths is a subset of an arithmetic progression}
	\mathsf{L}(x) \subseteq \left\{ l + k\cdot|\mathsf{n}(b) - \mathsf{d}(b)| : k\in\nn_0 \right\}.
	\end{equation}
	Now suppose by contradiction that $|\Delta(x)| > 1$. Since the inclusion~\ref{inclusion: set of lengths is a subset of an arithmetic progression} holds, there exists $t \in\nn_{>1}$ such that $t\,|\mathsf{n}(b) - \mathsf{d}(b)| \in\Delta(x)$. This implies that there exist $z,z' \in\mathsf{Z}(x)$ such that $|z| - |z'| = t\,|\mathsf{n}(b) - \mathsf{d}(b)|$ and $\mathsf{L}(x) \cap [|z'|,|z|] = \{|z'|, |z|\}$. We have two possible cases, either $|z_h| \leq |z'|$ or $|z| \leq |z_h|$. Assume that $|z_h| \leq |z'|$. In virtue of Lemma~\ref{lemma: any factorization of an element can be transformed into the hub factorization of that element}, there exist factorizations $z = z''_1, \ldots, z''_s = z_h \in \mathsf{Z}(x)$ such that the equality $\big||z''_i| - |z''_{i + 1}|\big| = |\mathsf{n}(b) - \mathsf{d}(b)|$ holds for all $i \in\llbracket 1,s-1 \rrbracket$. This implies that there exists $j \in \llbracket 1,s \rrbracket$ such that $|z'| < |z''_j| < |z|$, but this is a contradiction. On the other hand, if $|z| \leq |z_h|$ then, as before, there exist factorizations $z' = z^*_1, \ldots, z^*_r = z_h \in \mathsf{Z}(x)$ such that $\big||z^*_i| - |z^*_{i + 1}|\big| = |\mathsf{n}(b) - \mathsf{d}(b)|$ for all $i \in\llbracket 1,r-1 \rrbracket$ by Lemma~\ref{lemma: any factorization of an element can be transformed into the hub factorization of that element}. Again, this implies that there exists $j \in \llbracket 1,r \rrbracket$ such that $|z'| < |z^*_j| < |z|$, but this is a contradiction. Hence $|\Delta(x)| \leq 1$, and our proof concludes.
\end{proof}


Note that Theorem~\ref{theorem: characterization of sets of lengths of canonical multicyclic monoids} does not hold for all Puiseux monoids. The following example exhibits an atomic Puiseux monoid with an element whose set of lengths is not the union of finitely many MAPs.

\begin{example}
	Let $M = \langle 1/p \mid p\in\mathbb{P} \rangle$. It is not hard to see that $M$ is atomic with $\mathcal{A}(M) = \{1/p \mid p\in\mathbb{P}\}$. Moreover, $\mathsf{L}(1) = \mathbb{P}$. Since arbitrarily large prime gaps exist, $\mathsf{L}(1)$ is not the union of finitely many MAPs.
\end{example}

Next we prove that nontrivial canonical rational multicyclic monoids satisfy the Structure Theorem for Unions of Sets of Lengths (as stated in \cite[Theorem~4.2]{GaGe09}). 

\begin{proposition}
	Let $M_{\mathcal{B}}$ be a canonical rational multicyclic monoid such that $\mathcal{B} \neq \emptyset$. Then there exist constants $N,K \in\nn$ such that for all $k \geq K$ we have that\, $\mathcal{U}_k(M_{\mathcal{B}})$ is an AAP with difference $d = \min \Delta(M_{\mathcal{B}})$ and bound $N$.
\end{proposition}
\begin{proof}
	Our first goal is to show that $\Delta(M_{\mathcal{B}})$ is a nonempty finite set. For this purpose, let $D = \max_{b \in\mathcal{B}}\{|\mathsf{n}(b) - \mathsf{d}(b)|\}$ which is well defined given that $\mathcal{B} \neq \emptyset$, and consider a factorization $z = \sum_{i = 0}^{n} c_i b_i^{e_i} \in \mathsf{Z}(x)$ of a nonzero element $x$ with $n,c_i,e_i \in\nn_0$ and $b_i \in \mathcal{B}$ for every $i \in \llbracket 0,n \rrbracket$. If the inequalities $\mathsf{n}(b_i) \leq c_i$ and $b_i < 1$ hold for some $i \in \llbracket 0,n \rrbracket$ then by using the identity $\mathsf{n}(b_i) b_i^{e_i} = \mathsf{d}(b_i)b_i^{e_i + 1}$ we can generate a factorization $z' \in \mathsf{Z}(x)$ such that $0 < |z'| - |z| \leq D$. Similarly, if the inequalities $c_i \geq \mathsf{d}(b_i)$ and $b_i > 1$ hold for some $i \in \llbracket 1,n \rrbracket$ then there exists $z^* \in \mathsf{Z}(x)$ satisfying that $0 < |z^*| - |z| \leq D$. On the other hand, if $c_i < \min\{\mathsf{n}(b_i), \mathsf{d}(b_i)\}$ for $i \in \llbracket 1,n \rrbracket$ and $c_0 < \mathsf{n}(b)$ for all $b \in \mathcal{B}_{<1}$ then $z \in \mathsf{Z}(x)$ is the factorization of maximum length of $x$ by Remark~\ref{remark: we can transform the hub factorization into any other factorization}. Hence $\Delta(M_{\mathcal{B}})$ is a finite set. Note that $\Delta(M_{\mathcal{B}}) \neq \emptyset$ since $\mathcal{B} \neq \emptyset$. In virtue of \cite[Proposition 4.9]{ScGG2019}, there exists $k \in\nn$ such that $|\mathcal{U}_k(M_\mathcal{B})| = \infty$ which, in turn, implies that $\rho_k(M_{\mathcal{B}}) = \infty$. Our result follows after \cite[Theorem 4.2]{GaGe09}.
\end{proof}

\subsection{Arithmetic progressions as sets of distances}

As we mentioned before, determining the set of distances of a given monoid is, in general, a difficult task. In~\cite[Corollary 4.8]{BCKR06} Bowles et al. proved that, for $n,d$ positive integers, every set of the form $\{d, 2d, \ldots, nd\}$ occurs as the set of distances of some numerical monoid, while Geroldinger and Zhong proved in~\cite[Theorem 4.1]{GeQZ2019} that the set of distances of a transfer Krull monoid over an abelian group is an interval. We conclude this paper by proving that, for $d$ a positive integer, every set of the form $\{nd \mid n\in\nn\}$ occurs as the set of distances of some Puiseux monoid.  

\begin{proposition}
	Let $d$ be a positive integer. Then there exists a Puiseux monoid $M$ such that $\Delta(M) = \{md \mid m \in\nn\}$.
\end{proposition}

\begin{proof}
	Let $p_2, p_3, \ldots $ be a sequence of different prime numbers such that $(p_n - dn)_{n \in\nn_{\geq 2}}$ is increasing. We can assume without loss of generality that $p_n > dn + 1$ for all $n \in\nn_{\geq 2}$. Now consider the Puiseux monoid $M \coloneqq \langle b^s \mid b \in\mathcal{B},\, s \in\nn_0 \rangle$, where
	\[
		\mathcal{B} = \left\{\frac{p_{2k} - 2dk}{p_{2k}},\,\, \frac{p_{2k + 1} - 2dk + d}{p_{2k + 1}} \,\, \bigg|\,\,  k\in\nn\right\}.
	\]
First, $M$ is atomic by Proposition~\ref{prop: sufficient condition for being an atom}. Second, a nonzero element $x \in M$ has exactly one factorization of minimum length $z_h(x)$ by Lemma~\ref{lemma: factorization of minimum length when all primitive generators are proper fractions}. 
Third, given a nonzero element $x \in M$, there exists a canonical rational multicyclic monoid $M_{\mathcal{B'}}$ with $\mathcal{B'} \subseteq \mathcal{B}$ such that $x \in M_{\mathcal{B'}}$ and $\mathsf{Z}_M(x) = \mathsf{Z}_{M_{\mathcal{B'}}}(x)$. Indeed, if some positive power of an element $b \in\mathcal{B}$ divides $x$ in $M$ then either $\mathsf{n}(b)\, |_M\, x$ or some positive power of $b$ shows in the factorization $z_h(x)$ by Remark~\ref{remark: when all primitive generators are proper fractions there is no loss in assuming that the monoid is a canonical rational multicyclic monoid}, but the numerators of the elements of $\mathcal{B}$ formed an increasing sequence. This last property implies that, when analyzing $\mathsf{L}(x)$ for a fixed $x \in M$, there is no loss in assuming that $M$ is a canonical rational multicyclic monoid. 
Now let 
	\[
		z_k = (p_{2k} - 2dk)\left(\frac{p_{2k} - 2dk}{p_{2k}}\right)^2 + (p_{2k + 1} - 2dk + d)\left(\frac{p_{2k + 1} - 2dk + d}{p_{2k + 1}}\right)^2 \in\mathsf{Z}(x_{2k}),
	\]
where $k\in\nn$. Now fix $k \in \nn$. In virtue of Lemma~\ref{lemma: factorization of minimum length when all primitive generators are proper fractions}, $z_k$ is the factorization of minimum length of $x_{2k}$. Moreover, $\mathsf{L}(x_{2k}) = \{|z_k| + k_1(2dk) + k_2d(2k - 1) \mid k_1, k_2 \in\nn_0\}$ by the argument used in the first part of the proof of Theorem~\ref{theorem: characterization of sets of lengths of canonical multicyclic monoids}. Clearly, $d(2k - 1) \in\Delta(x_{2k})$. Now let $h \in \llbracket 1,2k-2 \rrbracket$, and take $z_1, z_2 \in \mathsf{Z}(x_{2k})$ such that $|z_1| = |z| + h(2dk)$ and $|z_2| = |z| + d(h + 1)(2k	- 1)$. Suppose, by way of contradiction, that there exists $z_3 \in\mathsf{Z}(x_{2k})$ such that $|z_1| < |z_3|< |z_2|$. Thus,
	\[
		h(2dk) < t_0(2dk) + s_0d(2k - 1) < d(h + 1)(2k - 1)
	\]
	for some nonnegative integers $t_0$ and $s_0$. After some computations, the first inequality yields that $h < t_0 + s_0$, and from the second inequality we obtain $t_0 + s_0 < h + 1$. This contradiction proves that our hypothesis is untenable. Hence $d(2k - h - 1) \in\Delta(x_{2k})$ for $h \in \llbracket 1,2k-2 \rrbracket$. In other words, we have $\{d, 2d, \ldots, (2k - 1)d\} \subseteq \Delta(x_{2k})$ for all $k\in\nn$. Consequently, the inclusion $\{md \mid m\in\nn\} \subseteq \Delta(M)$ holds. Conversely, let $z, z' \in \mathsf{Z}(x)$ be two factorizations of a nonzero element $x$ such that $|z| - |z'|$ is a distance of $x$. In virtue of Lemma~\ref{lemma: factorization of minimum length when all primitive generators are proper fractions}, there exist factorizations $z = z_1, \ldots, z_n = z_h \in \mathsf{Z}(x)$ such that $z_h$ is the factorization of minimum length of $x$ and $\big||z_i| - |z_{i + 1}|\big| = |\mathsf{n}(b) - \mathsf{d}(b)|$ for some $b \in \mathcal{B}$ and all $i \in\llbracket 1,n-1 \rrbracket$. Consequently, we have that $d$ divides $|z| - |z_h|$. Similarly, $d$ divides $|z'| - |z_h|$, which implies that the elements of $\Delta(x)$ are divisible by $d$ for all nonzero $x \in M$. This, in turn, implies that $\Delta(x) \subseteq \{md \mid m\in\nn\}$. Therefore, our result follows.
\end{proof}

	\section{Acknowledgments}
	
	The author wants to thank Felix Gotti for his guidance throughout the different stages of this manuscript and for many useful conversations about factorization theory. The author extends his thanks to anonymous referees whose feedback improve the final version of this paper. While working on this manuscript, the author was supported by the University of Florida Mathematics Department Fellowship.

\end{document}